\documentclass[3p]{elsarticle}
\usepackage[utf8x]{inputenc}
\usepackage{amsmath,amsthm}
\usepackage{amsfonts,amstext,amssymb, mathtools, comment}
\RequirePackage[colorlinks,citecolor=blue,urlcolor=blue]{hyperref}

\newcommand{\levy}{L\'{e}vy }
\newcommand{\p}{{\mathbb P}}
\newcommand{\e}{{\mathbb E}}
\newcommand{\D}{{\mathrm d}}

\newcommand{\1}[1]{\mbox{\rm\large  1}_{\{#1\}}}

\newcommand{\bs}{\boldsymbol}
\newcommand{\bss}{\bs s}
\newcommand{\bX}{\bs X}
\newcommand{\R}{\mathbb R}
\newcommand{\bx}{\bs x}
\newcommand{\bc}{\bs c}
\newcommand{\inner}[1]{\left\langle #1\right\rangle}

\renewcommand{\l}{\lambda}

\newcommand{\ii}{{\mathrm i}}

\newtheorem{theorem}{Theorem}
\newtheorem{cor}{Corollary}

\newtheorem{prop}{Proposition}
\newtheorem{remark}{Remark}

\begin{document}
\title{A bivariate risk model with mutual deficit coverage}
\author{Jevgenijs Ivanovs}
\ead{jevgenijs.ivanovs@unil.ch}
\address{Department of Actuarial Science, University of Lausanne, CH-1015 Lausanne, Switzerland}
\author{Onno Boxma}
\ead{o.j.boxma@tue.nl}
\address{Department of Mathematics and Computer Science, Eindhoven University of Technology, 5600 MB Eindhoven, The Netherlands}
\begin{abstract}
We consider a bivariate Cram\'er-Lundberg-type risk reserve process with the special feature that each insurance company agrees to cover the deficit of the other. 
It is assumed that the capital transfers between the companies are instantaneous and
incur a certain proportional cost, and that ruin occurs when neither company can cover the deficit of the other.
We study the survival probability as a function of initial capitals and express its bivariate transform
through two univariate boundary transforms, where one of the initial capitals is fixed at~0.
We identify these boundary transforms in the case when claims arriving at each company form two independent processes.
The expressions are in terms of Wiener-Hopf factors associated to two auxiliary compound Poisson processes.
The case of non-mutual agreement is also considered. The proposed model shares some features of a contingent surplus note instrument and may be of interest in the context of crisis management.
\end{abstract}
\maketitle

\section{Introduction}
Insurance companies cannot operate in isolation from financial markets or from other insurance and reinsurance companies.
Hence it is important to understand the effect of interaction on the main characteristics of an insurance company.
Multivariate risk models, however, present a serious mathematical challenge with few explicit results up to date, see~\cite[Ch.\ XIII.9]{asmussen2010ruin}.
This paper focuses on a nonstandard, but rather general bivariate risk model 
and provides an exact analytic study of the corresponding survival probability,
borrowing some ideas from the analysis of a somewhat related queueing problem in~\cite{boxma_ivanovs}. The proposed model may be of interest in the context of crisis management due to its relation to contingent surplus notes and mutual insurance. Moreover, it allows for an explicit structural result without imposing overly strict assumptions such as proportional or dominating claims, see Section~\ref{sec:literature}.

We consider a bivariate Cram\'er-Lundberg risk process as a model of surplus of two insurance companies (or two lines of one insurance business).
The special feature of our model is that the companies have a mutual agreement
to cover the deficit of each other.
More precisely, if company $1$ gets ruined, with its capital decreasing to a value $-x<0$,
then company $2$ compensates this deficit, bringing the capital of company $1$ back to $0$.
However, this comes at a price; a unit of capital received by company $1$ requires $r_1 \geq 1$ from company $2$ (cf.\ Figure 1).
If this would cause the capital of company $2$ to go below $0$, then both companies are said to be ruined.
Similarly, if company $2$ gets a deficit $-y<0$, company $1$ compensates this deficit, but its capital
reduces by $r_2 y \geq y$;
and if this would cause the capital of company $1$ to go below $0$, then again
both companies are said to be ruined.
Finally, ruin may also be caused by a single event bringing the surplus processes of both companies below~0. It may be more realistic to assume that if one company can not save the other from ruin then it does not transfer any capital at that instant and continues to operate. Note, however, that survival of both companies in this set-up corresponds to our previous notion of survival. 

\begin{figure}[h!]
  \centering
  \includegraphics{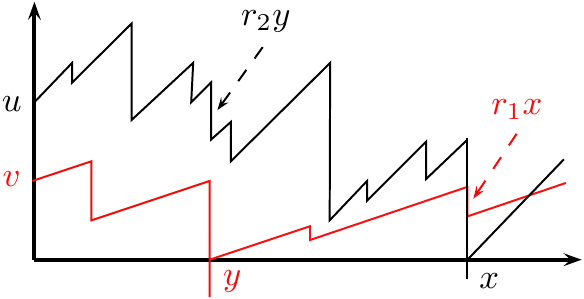}
  \caption
  {Illustration of mutual deficit coverage (no common shocks).}
\label{fig}
\end{figure}

Our main goal in this study is to provide an exact analysis of the (Laplace transform of the) probability of survival (i.e., ruin never occurs) $\phi(u,v)$, as a function
of the vector $(u,v)$ of initial capitals.
We do this by (i) expressing the two-dimensional Laplace transform $F(s_1,s_2)$ of $\phi(u,v)$
in terms of the transforms of $\phi(u,0)$ and $\phi(0,v)$,
and (ii) determining the latter two transforms by solving a Wiener-Hopf  boundary value problem in the case of independent claim streams.
In the latter step, a key role is played by the Wiener-Hopf factorization of two auxiliary compound Poisson processes.

In our terminology, if one company cannot save the other then both are declared ruined,
whereas in practice it is more likely that the company with a positive capital pays out all it has, but then continues its operation.
Hence it may be more appropriate to call $\phi$ the probability of survival of {\em both} companies, or the probability that external help is never needed, see
 Remark~\ref{Eremark}.
One may also notice that it is clearly better (with respect to survival) to merge the two lines eliminating transaction costs.
There may be cases, however, where a merger is not possible due to legal, regulatory or other issues.
Outside of an insurance context, the two lines may be two separate physical entities such as water reservoirs or energy sources.
Furthermore, one may consider the case where $r_i<1$ for at least one company, see Remark~\ref{rem:less1}, so that merging may not be optimal.
This may correspond to the case where part of the deficit is written off or is covered by some other fund.

Our model resembles two insurance companies with a mutual insurance fund, which
is used to cover the deficit of each of them. In our case, however, there is no separate fund - it is the other company which provides the capital to cover the deficit.
Another related notion is that of a contingent surplus note which is a form of a CoCo bond issued by an insurance company. This bond pays a higher coupon, because of the risk that it is converted into surplus if a trigger event (e.g.\ deficit) occurs. Finally, we may consider an insurance company and a governmental fund which is used to save companies from default. In this case capital transfers are only possible from the fund to the company; this particular scenario is discussed in Section~\ref{sec:restricted}.

The above-sketched risk model bears some resemblance to a two-dimensional queueing model
of two coupled processors.
This model features two $M/G/1$ queues, each of which, in isolation, is known to be the dual
of a Cram\'er-Lundberg insurance risk model, cf.\ \cite[Ch.\ I.4a]{asmussen2010ruin}.
Just like in our risk model, the two processors are coupled by the agreement to help each other.
When one of the two $M/G/1$ queues becomes empty,
the service speed of the other server -- say, server $i$ -- increases from $1$ to $r_i \geq 1$, i.e. server $i$ is being helped by the idle server.
It should be stressed that this similarity is rather loose and that there is no clear duality relation between these bivariate risk and queueing models.
Strikingly, the crucial ideas of the analysis of the coupled processor model in~\cite{boxma_ivanovs} apply to our present setup as well.
Moreover, solutions to both problems are based on the same Wiener-Hopf factors, which hints that there might be a certain duality relation between the two.

\subsection{Related literature}\label{sec:literature}
Despite their obvious relevance,
exact analytic studies of multidimensional risk reserve processes are scarce in the insurance literature.
A special, important case is the setting of proportional reinsurance,
which was studied in Avram et al. \cite{avram2008two}. There it is assumed that there
is a single arrival process, and the claims are proportionally split among two reserves.
In this case, the two-dimensional exit (ruin) problem becomes a {\em one}-dimensional
first-passage problem above a piece-wise linear barrier. Badescu et al.~\cite{badescu2011two}
have extended this model by allowing a dedicated arrival stream of claims into
only one of the insurance lines. They show that the transform of the time to ruin of at least
one of the reserve processes can be derived using similar ideas as in \cite{avram2008two}.

An early attempt to assess multivariate risk measures can be found in Sundt~\cite{Sundt}, where  multivariate Panjer recursions are developed
which are then used to compute the distribution of the aggregate claim process, assuming simultaneous claim events and discrete claim sizes.
Other approaches are deriving integro-differential equations for the various measures of risk
and then iterating these equations to find numerical approximations \cite{chan2003some,gong2012recursive},
or computing bounds for the different types of ruin probabilities that can occur in a setting
where more than one insurance line is considered, see~\cite{cai2005multivariate, cai2007dependence}.
In an attempt to solve the integro-differential equations that arise from such models, Chan et al.~\cite{chan2003some}
derive a Riemann-Hilbert boundary value problem for the bivariate Laplace transform
of the joint survival function (see~\cite{paper2} for details about such problems arising in the context
of risk and queueing theory and the book \cite{cohen_boxma} for an extended analysis of similar models in queueing).
However, this functional equation for the Laplace transform is not solved in \cite{chan2003some}.
In~\cite{paper2} a similar functional equation is taken as a departure point,
and it is explained how one can find transforms of ruin related performance measures
via solutions of the above mentioned boundary value problems.
It is also shown that the boundary value problem has an explicit solution in terms of transforms,
if the claim sizes are ordered. In \cite{paper3} this is generalized to the case in which the claim amounts
are also correlated with the time elapsed since the previous claim arrival.

Bivariate models where one company can transfer its capital to the other have already been considered in the literature.
Recently, Avram and Pistorius~\cite{florin} proposed a model of an insurance company which splits its premiums between a reinsurance/investement fund
and a reserves fund necessary for paying claims. In their setting only the second fund receives claims, and hence all capital transfers are one way: from the first fund to the second.
Another example is a capital-exchange agreement from~\cite[Ch.\ 4]{lautscham_thesis}, where two insurers pay dividends according to a barrier strategy
and the dividends of one insurer are transferred to the other unless the other is also fully capitalized.
This work resulted in systems of integro-differential equations for the expected time of ruin and expected discounted dividends,
which are hard to solve even in the case of exponential claims.

Finally, we briefly list related contributions in the queueing context.
The joint queue length distribution of the coupled processor model has been derived by Fayolle and Iasnogorodski \cite{FI},
in the case that the service time distributions at both queues are exponential. In their pioneering paper,
they showed how the generating function of the joint steady-state queue length distribution can be
obtained by solving a Riemann-Hilbert boundary value problem.
Cohen and Boxma \cite{cohen_boxma} generalized this queueing model by allowing general service time distributions.
They obtained the Laplace-Stieltjes transform of the joint steady-state workload distribution
by solving a Wiener-Hopf boundary value problem.
In \cite{boxma_ivanovs} the model of \cite{cohen_boxma} was extended by considering a pair of coupled queues driven by independent
spectrally positive L\'evy processes and a compact solution was obtained. There the model was also linked to a two-server fluid network.


\subsection{Organization of the paper}
In Section~\ref{Model} we describe the model in detail.
In Section~\ref{Survival} we derive an integral equation for the survival probability $\phi(u,v)$, as a function
of the vector $(u,v)$ of initial amounts of capital.
Section~\ref{Kernel} is devoted to the derivation of a so-called kernel equation for the two-dimensional Laplace transform
of $\phi(u,v)$, see Proposition~\ref{prop:kernel}. After a brief discussion of the net profit condition, in Section~\ref{Profit},
we solve the kernel equation in the case of independent claim streams in Section~\ref{Solution}. The main result is given in Theorem~\ref{thm:main},
which is then illustrated by two simple examples.
In Section~\ref{sec:restricted} we specialize to a model of non-mutual insurance, where help is provided by only one company (or government).
Some open problems are discussed in Section~\ref{Open}.

\section{The model}
\label{Model}

Consider a bivariate Cram\'er-Lundberg risk process $(u+X_1(t),v+X_2(t)),t\geq 0$ with initial capitals $u,v>0$, premium rates $c_i>0$,
claim arrival rate $\lambda$ and a joint claim size distribution $\mu(\D x_1,\D x_2)$.
In other words,
\[(X_1(t),X_2(t))=(c_1,c_2)t-\sum_{i=1}^{N(t)}(C_{1,i},C_{2,i}),\]
where $(C_{1,i},C_{2,i}),i\in\mathbb N$ are iid distributed according to~$\mu$, and $N(t)$ is an independent Poisson process of rate~$\lambda$.
Without loss of generality we can assume that $\mu$ does not have mass at $(0,0)$.
Next, we define the corresponding bivariate Laplace exponent by
\begin{align}\label{eq:transform}
 &\psi(s_1,s_2):=\\&\log \e e^{s_1X_1(1)+s_2X_2(1)}=s_1c_1+s_2c_2-\lambda\left(1-\int_{\mathbb R^2_+} e^{-s_1x_1-s_2x_2}\mu(\D x_1,\D x_2)\right).\nonumber
\end{align}
\begin{remark}\label{rem:independent}
Note that this model incorporates the case of two independent risk processes with claim arrival rates $\lambda_i$ and claim size distributions $\mu_i(\D x)$ with the relation
\begin{align*}
 &\lambda=\lambda_1+\lambda_2, &\mu(\D x,\D y)=\frac{\lambda_1}{\lambda_1+\lambda_2}\mu_1(\D x)\delta_0(\D y)+\frac{\lambda_2}{\lambda_1+\lambda_2}\mu_2(\D y)\delta_0(\D x),
\end{align*}
where $\delta_0$ denotes the Dirac point mass at~0. This yields $\psi(s_1,s_2)=\psi_1(s_1)+\psi_2(s_2)$ with
\[\psi_i(s):=\log \e e^{sX_i(1)}=sc_i-\lambda_i\left(1-\int_{\mathbb R_+} e^{-sx}\mu_i(\D x)\right), ~~~i=1,2.\]
\end{remark}

Throughout this work we assume that the companies have a mutual agreement of deficit coverage:
when a claim brings the surplus of a company $i\in\{1,2\}$ below 0 then the other company $j=3-i$ transfers just enough capital to bring it to 0
so that company $i$ can continue its operation.
If company $j$ has not enough capital then ruin occurs.
We also assume that a capital transfer incurs some proportional cost: a unit of capital received by $i$ requires $r_i\geq 1$ capital from $j$.
It is noted that in the case of no transaction costs our risk problem reduces to a classical one-dimensional problem obtained by considering the total surplus, see Section~\ref{sec:nocost}.
\begin{remark}\label{rem:less1}
All of the results below hold for any $r_i\geq 0$ such that $r_1r_2\geq 1$.
Even the latter condition can be removed, except that then the second statement in~Proposition~\ref{prop:phi} may no longer be true.
\end{remark}

It is convenient to consider a more general model defined for all times $t\geq 0$,
where ruin is avoided using capital from external sources.
More precisely, if company $j$ has insufficient funds to save $i$ then it transfers to $i$ all its capital and the rest of the required capital is taken from the shareholders of company~$i$
(or some other external source).
Mathematically, we can describe the bivariate surplus process $(S_1(t),S_2(t))$ in the following way:
\begin{align}\label{eq:def}
S_1(t)=u+X_1(t)-r_2L_2(t)+L_1(t)+E_1(t),\nonumber\\
S_2(t)=v+X_2(t)-r_1L_1(t)+L_2(t)+E_2(t),
\end{align}
where $L_i(t)$ and $E_i(t)$ are the cumulative amounts of capital received by the $i$-th company from the other company and from an external source respectively; $L_i(0)=E_i(0)=0$.
Note that the time of ruin in the original model is given by
\begin{equation}\label{eq:tau}
 \tau=\inf\{t\geq 0:E_1(t)+E_2(t)>0\}.
\end{equation}
In this work we only address the probability of survival $\phi(u,v):=\p_{u,v}(\tau=\infty)$.
Note that in our setup the concepts of ruin and survival have unambiguous meaning,
which should be compared to various possible ruin concepts in standard bivariate risk models, see~\cite{chan2003some}.
In the following we provide two mathematical definitions of the processes involved in~\eqref{eq:def}:
A recursive definition and a definition as a modified Skorokhod problem.

\begin{remark}
\label{Eremark}
 We note that the external sources $E_i$ are introduced only for modeling convenience. Throughout this work survival means survival in the original model which
corresponds to never using external capital in the extended model.
Hence $E_i$ never appear in the following apart from the discussion of the net profit condition in Section~\ref{Profit} and direct calculations in Section~\ref{sec:nocost}.
Nevertheless, we believe that it is important to have an explicit model definition as given by~\eqref{eq:def} and~\eqref{eq:tau}, see also Section~\ref{Open}.
\end{remark}

\subsection{Recursive definition}\label{sec:recursive}
Let us remark that the processes $L_i,E_i$ are piece-wise constant and non-decreasing, and can have jumps at claim arrival times $t_1,t_2,\ldots$ only.
One may define these processes recursively by considering their jumps at time $t_n$ for $n=1,2,\ldots$
More precisely, with $x_i=S_i(t_n-)+(X_i(t_n)-X_i(t_n-))$ we have the following cases:
\begin{itemize}
 \item $x_1,x_2\geq 0$: $L_i,E_i$ have no jumps (at $t_n$),
\item $x_1<0,x_2+r_1x_1\geq 0$: $L_1$ has a jump $-x_1$ (and no other jumps at $t_n$),
\item $x_2<0,x_1+r_2x_2\geq 0$: $L_2$ has a jump $-x_2$,
\item $x_1<0,x_2\geq 0,x_2+r_1x_1<0$: $L_1$ has a jump $x_2/r_1$ and $E_1$ has a jump $-x_1-x_2/r_1$,
\item $x_2<0,x_1\geq 0,x_1+r_2x_2<0$: $L_2$ has a jump $x_1/r_2$ and $E_2$ has a jump $-x_2-x_1/r_2$,
\item $x_1<0,x_2<0$: $E_i$ has a jump $-x_i$ for $i=1,2$.
\end{itemize}

\subsection{Modified Skorokhod problem definition}\label{sec:skorokhod}
The following equivalent specification of our model closely resembles a multidimensional Skorokhod problem, see~\cite{harrison_rieman_reflected,kella_reflecting}
building upon Skorokhod's one dimensional construction~\cite{skorokhod_refl}.
It is noted that for the classical bivariate Skorokhod problem one must assume that $r_1r_2<1$ which is clearly violated in our setup.

We consider the equations in~\eqref{eq:def} with the following additional requirements:
\begin{itemize}
 \item $S_i\geq 0$ and $L_i,E_i$ are piece-wise constant and non-decreasing, 
\item the jump times of $L_i$ are contained in $\{t\geq 0:S_i(t)=0\}$,
\item the jump times of $E_i$ are contained in $\{t\geq 0:S_1(t)=S_2(t)=0\}$,
\item out of $L_1,L_2,E_1,E_2$ the only pairs which can jump simultaneously are $(L_1,E_1),(L_2,E_2)$ and $(E_1,E_2)$.
\end{itemize}
It is easy to check that these requirements are satisfied by the recursive construction from Section~\ref{sec:recursive},
and that no other choices of $L_i,E_i$ and thus $S_i$ satisfy these conditions.

A few words about the above requirements. The second states that a company receives capital when in need,
and the third states that external capital can be used only when both companies are in need (or one is at 0 after trying to save the other).
The final condition is necessary to prevent redundant capital transfers.

\section{The survival probability}
\label{Survival}

As stated in Remark~\ref{Eremark}, we now ignore the possibility of funding by external sources.
Consider $\phi(u,v)=\p_{u,v}(\tau=\infty)$, the probability of survival for the initial capitals $u,v\geq 0$, and put for convenience $\phi(u,v)=0$ if $u<0$ or $v<0$.
It is immediate from sample path comparison that $\phi(u,v)$ is non-decreasing in both $u,v$.
Furthermore, $\phi(u,v)$ is a (Lipschitz) continuous function on $\R^2_+$, which follows from
$\phi(u,v)\geq (1-\l h)\phi(u+c_1h,v+c_2h)+o(h)$ as $h\downarrow 0$ implying
\begin{equation}\label{eq:bound}
 \phi(u+c_1h,v+c_2h)-\phi(u,v)\leq \l h+o(h),
\end{equation}
where $o(h)$ does not depend on $u,v$.
Moreover, our model definition implies the following equation
\begin{align}\label{eq:lh}
\phi(u,v)&=o(h)+(1-\lambda h)\phi(u+c_1h,v+c_2h)\\
&+\l h \iint\limits_{x\in[0,u],y\in[0,v]}\phi(u-x,v-y)\mu(\D x,\D y)\nonumber\\
&+\l h \iint\limits_{x\in[0,u],y>v}\phi(u-x+r_2(v-y),0)\mu(\D x,\D y)\nonumber\\
&+\l h \iint\limits_{x>u,y\in[0,v]}\phi(0,v-y+r_1(u-x))\mu(\D x,\D y).\nonumber
\end{align}
In fact, this equation requires some careful considerations, and so its proof is given in the Appendix.

Equation~\eqref{eq:lh} implies that the directional derivative
\[D_{(c_1,c_2)}\phi(u,v)=\lim_{h\downarrow 0}\frac{\phi(u+c_1 h,v+c_2h)-\phi(u,v)}{h}\] exists and the following equation holds true.
\begin{align*}\phi(u,v)=\frac{1}{\l}D_{(c_1,c_2)}\phi(u,v)&+\iint\limits_{x\in[0,u],y\in[0,v]}\phi(u-x,v-y)\mu(\D x,\D y)\\
 &+ \iint\limits_{x\in[0,u],y>v}\phi(u-x+r_2(v-y),0)\mu(\D x,\D y)\nonumber\\
&+\iint\limits_{x>u,y\in[0,v]}\phi(0,v-y+r_1(u-x))\mu(\D x,\D y).
\end{align*}
\begin{remark}
 If the partial derivatives $\phi_1$ and $\phi_2$ exist and are continuous at $(u,v)$ then $D_{(c_1,c_2)}\phi(u,v)=c_1\phi_1(u,v)+c_2\phi_2(u,v)$.
This can not hold in general when $\mu$ has an atom at $(u,v)$, as can be seen from ~\eqref{eq:lh}.
\end{remark}

\section{The kernel equation}
\label{Kernel}
Equation~\eqref{eq:lh} can be equivalently formulated in terms of Laplace transforms. This results in a very simple identity given by the following Proposition.
\begin{prop}\label{prop:kernel}
Let $F(s_1,s_2):=\int_{\R_+^2} e^{-s_1u-s_2v}\phi(u,v)\D u\D v$ and
\begin{align*}
 &F_1(s_1)=\int_0^\infty e^{-s_1u}\phi(u,0)\D u, &F_2(s_2)=\int_0^\infty e^{-s_2v}\phi(0,v)\D v.
\end{align*}
Then for all $s_1,s_2> 0$ it holds that
 \begin{equation}\label{eq:kernel}
\psi(s_1,s_2)F(s_1,s_2)=\frac{\psi(s_1,s_2)-\psi(s_1,r_2s_1)}{s_2-r_2s_1}F_1(s_1)+\frac{\psi(s_1,s_2)-\psi(r_1s_2,s_2)}{s_1-r_1s_2}F_2(s_2).
\end{equation}
\end{prop}
This result identifies the transform of $\phi(u,v)$ up to the unknown functions $F_i(s_i)$.
Equation~\eqref{eq:kernel} resembles the so-called basic adjoint relation in the semimartingale reflected Brownian motion (SRBM) literature, see e.g.~\cite[Eq.~(3.2)]{williams_survey},
with the class of functions $f(u,v)=e^{-s_1u-s_2v}$.
In the queueing literature this type of functional equation is sometimes called `kernel equation', see e.g.~\cite[Ch.\ III.3]{cohen_boxma} and~\cite[Prop.\ 1]{boxma_ivanovs}.

\begin{proof}[Proof of Proposition~\ref{prop:kernel}]
We simply take the transform of~\eqref{eq:lh} and let $h \downarrow 0$.
Let us first consider the transform of $\phi(u+c_1h,v+c_2h)-\phi(u,v)$:
\begin{align*}
  &\int_{\R_+^2} e^{-s_1u-s_2v}\phi(u+c_1h,v+c_2h)\D u\D v-F(s_1,s_2)\\
  &=e^{s_1c_1h+s_2c_2h}\int_{\R_+^2} \1{u>c_1h,v>c_2h}e^{-s_1u-s_2v}\phi(u,v)\D u\D v-F(s_1,s_2)\\
&=(e^{s_1c_1h+s_2c_2h}-1)F(s_1,s_2)-\int_{\R_+^2}\1{u\leq c_1 h\text{ or }v\leq c_2 h} e^{-s_1u-s_2v}\phi(u,v)\D u\D v.
\end{align*}
By dominated convergence and continuity of $\phi$ we have
\begin{align*}
&\lim_{h\downarrow 0}\frac{1}{h}\int_{u\geq 0}\int_{v\leq c_2h} e^{-s_1u-s_2v}\phi(u,v)\D u\D v\\
&=
c_2\int_{u\geq 0}e^{-s_1u}\lim_{h\downarrow 0}\frac{1}{c_2h}\int_{v\leq c_2h} e^{-s_2v}\phi(u,v)\D u\D v=
c_2F_1(s_1).
\end{align*}
Similarly, we have
\[\lim_{h\downarrow 0}\frac{1}{h}\int_{u\leq c_1h}\int_{v\geq 0} e^{-s_1u-s_2v}\phi(u,v)\D u\D v=c_1F_2(s_2)\]
and $\int_{u\leq c_1h}\int_{v\leq c_2h} e^{-s_1u-s_2v}\phi(u,v)\D u\D v=o(h)$.
Combining these we obtain
\begin{equation}\label{eq:p1}
 \lim_{h\downarrow 0}\frac{1}{h}\int_{\R_+^2} (\phi(u+c_1h,v+c_2h)-\phi(u,v)) \D u\D v=(s_1c_1+s_2c_2)F(s_1,s_2)-c_2F_1(s_1)-c_1F_2(s_2).
\end{equation}

It is left to consider the transforms of the integrals in~\eqref{eq:lh}.
Using Fubini's theorem we get
\begin{align*}
 &\int_{\R_+^2} e^{-s_1u-s_2v}\int_{\R_+^2}\1{x\leq u,y\leq v}\phi(u-x,v-y)\mu(\D x,\D y)\D u\D v\\
&=\int_{\R_+^2}\int_{\R_+^2}\1{u\geq x,v\geq y}e^{-s_1u-s_2v}\phi(u-x,v-y)\D u\D v\mu(\D x,\D y)\\
&=\int_{\R_+^2}\int_{\R_+^2}e^{-s_1(u+x)-s_2(v+y)}\phi(u,v)\D u\D v\mu(\D x,\D y)=\hat G(s_1,s_2)F(s_1,s_2),
\end{align*}
where $\hat G(s_1,s_2)=\int_{\R_+^2}e^{-s_1x-s_2y}\mu(\D x,\D y)$.

Next, using the substitution $u'=u-x+r_2(v-y)$ we obtain
\begin{align*}
 &\int_{\R_+^2} e^{-s_1u-s_2v}\int_{\R_+^2}\1{x\leq u,y>v}\phi(u-x+r_2(v-y),0)\mu(\D x,\D y)\D u\D v\\
&=\int_{\R_+^2}\int_{\R_+^2}\1{u\geq x,v<y}e^{-s_1u-s_2v}\phi(u-x+r_2(v-y),0)\D u\D v\mu(\D x,\D y)\\
&=\int_{\R_+^2}e^{-s_1x-s_1r_2y}\int_{\R_+^2}\1{v<y}e^{-s_1u-(s_2-s_1r_2)v}\phi(u,0)\D u\D v\mu(\D x,\D y)\\
&=\int_{\R_+^2}e^{-s_1x-s_1r_2y}F_1(s_1)\frac{1}{s_2-s_1r_2}(1-e^{-(s_2-s_1r_2)y})\mu(\D x,\D y)\\
&=\frac{F_1(s_1)}{s_2-s_1r_2}(\hat G(s_1,s_1r_2)-\hat G(s_1,s_2)),
\end{align*}
and a similar equation for the last integral in~\eqref{eq:lh}.

Finally, we combine all the terms to get
\begin{align*}
 &0=(s_1c_1+s_2c_2)F(s_1,s_2)-c_1F_2(s_2)-c_2F_1(s_1)\\
&+\lambda\left(-F(s_1,s_2)+\hat G(s_1,s_2)F(s_1,s_2)\right.\\
&\left.+\frac{F_1(s_1)}{s_2-s_1r_2}(\hat G(s_1,s_1r_2)-\hat G(s_1,s_2))+
\frac{F_2(s_2)}{s_1-s_2r_1}(\hat G(s_2r_1,s_2)-\hat G(s_1,s_2))\right).
\end{align*}
Some simple manipulations using~\eqref{eq:transform} show that this equation coincides with~\eqref{eq:kernel} and so we are done.
\end{proof}

The following Corollary for two independent driving processes is immediate.
\begin{cor}\label{cor:kernel}
If $X_1$ and $X_2$ are independent processes given in Remark~\ref{rem:independent} then
\begin{equation}\label{eq:kernel_ind}
(\psi_1(s_1)+\psi_2(s_2))F(s_1,s_2)=\frac{\psi_2(s_2)-\psi_2(r_2s_1)}{s_2-r_2s_1}F_1(s_1)+\frac{\psi_1(s_1)-\psi_1(r_1s_2)}{s_1-r_1s_2}F_2(s_2)
\end{equation}
for all $s_1,s_2>0$.
\end{cor}

\section{The net profit condition}
\label{Profit}
Let $\mu_i=\e X_i(1)\in[-\infty,\infty)$ be the average drift of the driving process $X_i$.
In the following we assume that the net profit condition holds
\begin{align}\label{eq:net}
 &\mu_1+r_2\mu_2>0, &\mu_2+r_1\mu_1>0.
\end{align}
The importance of this condition is explained by the following result.
\begin{prop}\label{prop:phi}The following dichotomy is true for $r_i\in (0,\infty)$ with $r_1r_2\geq 1$:
\begin{itemize}
\item if~\eqref{eq:net} holds then $\phi(\infty,0)=\phi(0,\infty)=\phi(\infty,\infty)=1$,
\item otherwise $\phi(u,v)=0$ for all $u,v\in\R_+$.
\end{itemize}
\end{prop}
\begin{proof}
Equation~\eqref{eq:def} implies that
\[0\leq S_1(t)+r_2 S_2(t)=u+r_2v+X_1(t)+r_2X_2(t)+(1-r_1r_2)L_1(t)+E_1(t)+r_2 E_2(t).\]
Thus if $u+r_2v+X_1(t)+r_2X_2(t)<0$ then necessarily $E_1(t)+r_2 E_2(t)>0$ implying that $\tau\leq t$.
Note that if~\eqref{eq:net} does not hold then at least one of the Cram\'er-Lundberg processes $X_1(t)+r_2X_2(t)$ and $r_1X_1(t)+X_2(t)$
is certain to get ruined,
which shows that $\phi(u,v)=0$.

In the following assume that~\eqref{eq:net} holds, which guarantees that at least one of $\mu_i$ is positive. Without loss of generality we assume that~$\mu_1>0$.
By considering the case that the first company does not need any help we arrive at the bound \[\phi(u,0)\geq \p(\forall t\geq 0:u+X_1(t)+\underline X_2(t)r_2\geq 0).\]
Here $\underline X_2(t)$ is the running infimum of the process $X_2$.
It is well known in the queueing literature that a.s.~$X_1(t)/t\rightarrow \mu_1$ and $\underline X_2(t)/t\rightarrow\min(\mu_2,0)$ as $t\rightarrow \infty$.
Now we see from $\mu_1+r_2\mu_2>0$ that a.s.~$X_1(t)+\underline X_2(t)r_2\rightarrow\infty$ and hence the infimum of this process is finite.
This yields $\phi(\infty,0)=1$.

For any $\epsilon>0$ choose $u$ so large that $\phi(u,0)>1-\epsilon/2$, and consider the first passage time
$T=\inf\{t\geq 0:X_1(t)+\underline X_1(t)>u\}$ of the reflected process above level $u>0$.
In order to guarantee that $\phi(0,v)>1-\epsilon$ we need to choose $v$ so large that $\p(\underline X_2(T)+\underline X_1(T)r_1<-v)<\epsilon/2$,
i.e.\ the second company has enough capital for the first to reach high capital.
This is clearly possible, because $T<\infty$ a.s.
\end{proof}

\subsection{From survival function to measure}\label{sec:measures}
Since $\phi(u,v)\in[0,1]$ is continuous and non-decreasing in both $u,v\geq 0$ it defines a finite (probability) measure $\phi(\D u,\D v)$ on $\R_+^2$,
where $\phi(u,v)=\int_{x\in[0,u]}\int_{y\in[0,v]}\phi(\D x,\D y)$.
By Fubini we have for $s_1,s_2>0$ the following relation
\begin{align*}
 &\hat F(s_1,s_2):=\int_{\R_+^2} e^{-s_1u-s_2v}\phi(\D u,\D v)\\
&=s_1s_2\int_{\R_+^2} \int_{\R_+^2} \1{x\geq u,y\geq v}e^{-s_1x-s_2y}\D x\D y\phi(\D u,\D v)\\
&=s_1s_2\int_{\R_+^2} e^{-s_1x-s_2y} \phi(x,y) \D x\D y \\
=s_1s_2F(s_1,s_2).
\end{align*}
Similarly,
\begin{align*}
&\hat F_1(s):=\int_{0-}^\infty e^{-su}\phi(\D u,0)=s F_1(s),\\
&\hat F_2(s):=\int_{0-}^\infty e^{-sv}\phi(0,\D v)=s F_2(s)
\end{align*}
for $s>0$.
We will refer to $\phi(\D u,0)$ and $\phi(0,\D v)$ as boundary measures.
Note that one can easily rewrite the results of Proposition~\ref{prop:kernel} and Corollary~\ref{cor:kernel} in terms of measure transforms $\hat F,\hat F_1,\hat F_2$.
Finally, observe that
\[\lim_{s_1\rightarrow\infty}\hat F(s_1,s_2)=\hat F_2(s_2), \quad \lim_{s_2\rightarrow\infty}\hat F(s_1,s_2)=\hat F_1(s_1), \quad \lim_{s_1\rightarrow\infty,s_2\rightarrow\infty}\hat F(s_1,s_2)=\phi(0,0)\]
and hence given $F$ we can easily find $F_1,F_2$ and $\phi(0,0)$.

\section{Independent driving processes}
\label{Solution}
Throughout this section we assume that the claim streams are independent, and thus we can focus on the kernel equation in Corollary~\ref{cor:kernel}.
For this case we provide an explicit solution for the bivariate Laplace transform
$F(s_1,s_2)$ of the survival probability $\phi(u,v)$.

Even though the kernel equation is quite different from the kernel equation of a coupled processor model,
the method of analysis from~\cite{boxma_ivanovs}, building upon~\cite{cohen_boxma}, can still be used.
In the following we present a self-contained (apart from a few technical properties which can be found in~\cite{boxma_ivanovs}) application of this method to our risk problem.

We assume that $X_1$ and $X_2$ are two independent Cram\'er-Lundberg processes with exponents $\psi_1(s)$ and $\psi_2(s)$,
and so $\psi(s_1,s_2)=\psi_1(s_1)+\psi_2(s_2)$, see Remark~\ref{rem:independent}.
It is well known that $\psi_i(s)=q$ for $q>0$ has a unique positive solution, call it $\Phi_i(q)$.
Moreover, $\Phi_i(0)=\lim_{q\downarrow 0}\Phi_i(q)$ is a solution of $\psi_i(s)=0$, which is positive when $\mu_i<0$ and $\Phi_i(0)=0$ when $\mu_i\geq 0$.

It turns out that there exists a unique non-decreasing \levy process $Y_i$ (descending ladder time process~\cite[p.\ 170]{kyprianou} corresponding to $X_i$) such that
\begin{equation}\label{eqY}\psi^Y_i(q):=\log\e e^{-q Y_i(1)}=\mu_i^+-\frac{q}{\Phi_i(q)}, q>0,\end{equation}
(here and in the sequel, the positive and negative parts of some entity $c$ are denoted by $c^+:=\max(c,0)$ and $c^-:=-\min(c,0)$).
In our case $Y_i$ is a compound Poisson process (CPP) with no deterministic drift and jumps distributed as $\inf\{t\geq 0:X_i(t)<0\}$,
see also~\cite{boxma_ivanovs}.
Define two two-sided CPP processes $X_L$ and $X_R$ and two constants $p_L$ and $p_R$ as follows:
\begin{align}\label{eq:X_LR}
&X_L(t)=Y_1(r_1t)-Y_2(t), &X_R(t)=Y_1(t)-Y_2(r_2t),\nonumber\\[-8pt]\\[-8pt]
&p_L=\mu^+_2+r_1\mu^+_1, &p_R=\mu^+_1+r_2\mu^+_2,\nonumber
\end{align}
where we assume that $Y_1$ and $Y_2$ are independent.
Note that $p_L,p_R>0$ according to the net profit condition~\eqref{eq:net}.
Observe that the corresponding Laplace exponents are
given by
\begin{align}\label{eq:phi_phi}
&\psi_L(-w):=\log \e e^{-w X_L(1)}=p_L-r_1\frac{w}{\Phi_1(w)}+\frac{w}{\Phi_2(-w)},\\
&\psi_R(-w):=\log \e e^{-w X_R(1)}=p_R-\frac{w}{\Phi_1(w)}+r_2\frac{w}{\Phi_2(-w)},\label{eq:phi_phi2}
\end{align}
where $w \in \ii\R, w\neq0$.

Finally, we introduce the Wiener-Hopf factors for $X_L,p_L$:
\begin{align*}
&\Psi_L^+(w)=\e e^{-w\overline X_L(e_{p_L})}, ~~ \Re(w)\geq0, &\Psi_L
^-(w)=\e e^{-w\underline X_L(e_{p_L})}, ~~ \Re(w)\leq0,
\end{align*}
where $e_{p_L}$ denotes an independent exponential random variable of rate $p_L$ and $\overline X_L,\underline X_L$ are the
running supremum and infimum of $X_L$ respectively. It is well known that
\begin{align}\label{eq:WH}&\Psi_L^+(w)\Psi_L^-(w)=\frac{p_L}{p_L-\psi_L(-w)}, &w\in \ii\R,\end{align}
see e.g.~\cite[Thm.\ 6.16]{kyprianou}.
In a similar way we define the Wiener-Hopf factors for $X_R,p_R$.
We are now ready to formulate our main result.

\begin{theorem}\label{thm:main}
For independent $X_1$ and $X_2$ satisfying~\eqref{eq:net} and $s_i>\Phi_i(0)$ it holds that
  \begin{align*}
(\psi_1(s_1)+\psi_2(s_2))F(s_1,s_2)&=
\frac{p_R'}{s_2-r_2s_1}\frac{\psi_2(s_2)-\psi_2(r_2s_1)}{\psi_1(s_1)+\psi_2(r_2s_1)}\frac{\Psi_L^+(\psi_1(s_1))}{\Psi_R^+(\psi_1(s_1))}\\
&+\frac{p_L'}{s_1-r_1s_2}\frac{\psi_1(s_1)-\psi_1(r_1s_2)}{\psi_2(s_2)+\psi_1(r_1s_2)}\frac{\Psi_R^-(-\psi_2(s_2))}{\Psi_L^-(-\psi_2(s_2))},
\end{align*}
where
\begin{align*}&p_L'=r_1\mu_1+\mu_2^+-r_1r_2\mu_2^-, &p_R'=r_2\mu_2+\mu_1^+-r_1r_2\mu_1^-.\end{align*}
\end{theorem}
We note that the cases $s_2 = r_2 s_1$ and $s_1 = r_1 s_2$ should be understood in the limiting sense.
Importantly, numerical evaluation of the survival probability function $\phi(u,v)$ using
Theorem~\ref{thm:main} is feasible.
We refer an interested reader to~\cite{denIseger2,denIsiger}
for a discussion of an efficient and accurate numerical inversion of a two-dimensional Laplace transform and computation of the Wiener-Hopf factors, respectively. 

\begin{proof}[Proof of Theorem~\ref{thm:main}]
We rewrite~\eqref{eq:kernel_ind} using measure transforms from Section~\ref{sec:measures}:
\begin{align*}
 &(s_1-r_1s_2)(s_2-r_2s_1)(\psi_1(s_1)+\psi_2(s_2))\hat F(s_1,s_2)\\
&=(s_1-r_1s_2)(\psi_2(s_2)-\psi_2(r_2s_1))\hat F_1(s_1)s_2\\
&+(s_2-r_2s_1)(\psi_1(s_1)-\psi_1(r_1s_2))\hat F_2(s_2)s_1.
\end{align*}
By analytic continuation and continuity on the imaginary axis we see that this equation holds for all $s_1,s_2$ with $\Re(s_1),\Re(s_2)\geq 0$.
It is shown in~\cite{boxma_ivanovs} that $\Phi_i(s)$ can be analytically continued to $\Re(s)\geq 0$.
Moreover, in this domain $\Re(\Phi_i(s))>0$ for $\Re(s)>0$, $\Phi_i(s)\neq 0$ for $s\neq 0$, and the identity $\psi_i(\Phi_i(s))=s$ is preserved.
Thus we can plug in the above equation $s_1=\Phi_1(w),s_2=\Phi_2(-w)$ for $w\in i\mathbb R$ to obtain the following:
\begin{align*}
0&=(\Phi_1(w)-r_1\Phi_2(-w))(-w-\psi_2(r_2\Phi_1(w)))\hat F_1(\Phi_1(w))\Phi_2(-w)\\
&+(\Phi_2(-w)-r_2\Phi_1(w))(w-\psi_1(r_1\Phi_2(-w)))\hat F_2(\Phi_2(-w))\Phi_1(w).
\end{align*}
Assume for a moment that $w\neq 0$ and multiply both sides by $w/(\Phi_1(w)\Phi_2(-w))^2$ which using~\eqref{eq:phi_phi} leads to
\begin{align*}
&(p_L-\psi_L(-w))\frac{-w-\psi_2(r_2\Phi_1(w))}{\Phi_1(w)}\hat F_1(\Phi_1(w))=\\
&(p_R-\psi_R(-w))\frac{w-\psi_1(r_1\Phi_2(-w))}{\Phi_2(-w)}\hat F_2(\Phi_2(-w)).
\end{align*}
Using Wiener-Hopf factorization we arrive at
\begin{align}\label{eq:entire}
&p_L\frac{\Psi_R^+(w)}{\Psi_L^+(w)}\frac{w+\psi_2(r_2\Phi_1(w))}{\Phi_1(w)}\hat F_1(\Phi_1(w))=\\
&p_R\frac{\Psi_L^-(w)}{\Psi_R^-(w)}\frac{-w+\psi_1(r_1\Phi_2(-w))}{\Phi_2(-w)}\hat F_2(\Phi_2(-w)).\nonumber
\end{align}

Consider~\eqref{eq:entire} and note that the lhs is analytic in $\Re(w)>0$, the rhs is analytic in $\Re(w)<0$, and both are continuous and coincide on the imaginary axis;
for $w=0$ this is checked by taking the limit in the respective half plane, see also~\eqref{eq:limit} below.
Hence this equation defines an entire function.
Let us show that it is bounded and hence is a constant, call it $C$, by Liouville's theorem.
According to~\cite{boxma_ivanovs} the ratios of W-H factors are bounded in their respective half planes.
The transforms $\hat F_i$ of boundary measures are bounded by~1.
Finally we consider $(w+\psi_2(r_2\Phi_1(w)))/\Phi_1(w)$ for $\Re(w)\geq 0$; the corresponding term on the rhs can be analyzed in the same way.
Boundedness follows from the following simple observations:
$|\Phi_1(w)|\rightarrow\infty,\psi_2(w)/w\rightarrow c_2,\psi_1(\Phi_1(w))/\Phi_1(w)=w/\Phi_1(w)\rightarrow c_1$ as $|w|\rightarrow\infty$.

Both sides of \eqref{eq:entire} are equal to some constant~$C$, which can be identified by taking the limit $w\rightarrow 0$.
Suppose $\mu_1\geq 0$ then $\Phi_1(0)=0$ and hence we have
\begin{equation}\label{eq:limit}
 \lim_{|w|\downarrow 0,\Re(w)\geq 0}\frac{w+\psi_2(r_2\Phi_1(w))}{\Phi_1(w)}=\psi_1'(0)+r_2\psi_2'(0)=\mu_1+r_2\mu_2.
\end{equation}
This results in
\[C=p_L(\mu_1+r_2\mu_2)\phi(\infty,0),\]
and for $\mu_2\geq 0$ we similarly have
\[C=p_R(\mu_2+r_1\mu_1)\phi(0,\infty).\]
Recall that under the net profit condition~\eqref{eq:net}, we have $\phi(\infty,0)=\phi(0,\infty)=1$ and at least one of $\mu_1,\mu_2$ should be positive.
Now it is not hard to check that $C/p_L=p_R', C/p_R=p_L'$.

Finally, pick $s_i> \Phi_i(0)$ which implies $\Phi_i(\psi_i(s_i))=s_i$.
By considering the left part of~\eqref{eq:entire} for $w=\psi_1(s_1)>0$ and its right part for $w=-\psi_2(s_2)<0$ we obtain
\begin{align*}
 F_1(s_1)&=\frac{C}{p_L(\psi_1(s_1)+\psi_2(r_2s_1))}\frac{\Psi_L^+(\psi_1(s_1))}{\Psi_R^+(\psi_1(s_1))},\\
 F_2(s_2)&=\frac{C}{p_R(\psi_2(s_2)+\psi_1(r_1s_2))}\frac{\Psi_R^-(-\psi_2(s_2))}{\Psi_L^-(-\psi_2(s_2))}.
\end{align*}
Combining these calculations with Corollary~\ref{cor:kernel} we get the result.
\end{proof}

\begin{cor}\label{cor:zero_init}
 For independent $X_1$ and $X_2$ satisfying~\eqref{eq:net} it holds that
\[\phi(0,0)=\frac{p_R'}{c_1+r_2c_2}\frac{\Psi_L^+(\infty)}{\Psi_R^+(\infty)}=\frac{p_L'}{c_1+r_2c_2}\frac{\Psi_R^-(-\infty)}{\Psi_L^-(-\infty)}.\]
\end{cor}
\begin{proof}
First observe that $\Psi_L^+(\infty)=\p(\overline X_L(e_{p_L})=0)$ and similar observations hold true for the other W-H factors.
 From Theorem~\ref{thm:main} we see that
  \begin{align*}
\hat F(s_1,s_2)&=
\frac{p_R's_2}{s_2-r_2s_1}\left(\frac{s_1}{\psi_1(s_1)+\psi_2(r_2s_1)}-\frac{s_1}{\psi_1(s_1)+\psi_2(s_2)}\right)\frac{\Psi_L^+(\psi_1(s_1))}{\Psi_R^+(\psi_1(s_1))}\\
&+\frac{p_L's_1}{s_1-r_1s_2}\left(\frac{s_2}{\psi_2(s_2)+\psi_1(r_1s_2)}-\frac{s_2}{\psi_1(s_1)+\psi_2(s_2)}\right)\frac{\Psi_R^-(-\psi_2(s_2))}{\Psi_L^-(-\psi_2(s_2))}.
\end{align*}
To get the first statement, let $s_2\rightarrow\infty$ and then $s_1\rightarrow\infty$.
The second is obtained by reversing the order of limit operations.
\end{proof}

\subsection{Example: contingent surplus note}
Let us provide a check for a simple particular case, which admits a direct analysis.
Suppose that $X_2(t)=c_2 t$, i.e.~the second company is, in fact, a surplus note, which earns a constant premium and provides a cover for the first company.
We also assume that the net profit condition~\eqref{eq:net} holds.

Note that our system survives if and only if the one dimensional Cram\'er-Lundberg process $Z(t):=X_1(t)+(\mu_2/r_1) t$ survives when started at $u+v/r_1$.
Hence, with $\phi_Z(x)$ the survival probability for the $Z$ process when started at $x$,
\begin{align}
&F(s_1,s_2)=\int_{\R_+^2}e^{-s_1u-s_2v}\phi_Z(u+v/r_1)\D u\D v\nonumber\\
&=\int_0^\infty e^{-s_1 u}\phi_Z(u)\int_0^\infty \1{u>v/r_1}e^{-(s_2-s_1/r_1)v}\D v\D u\nonumber\\
&=\frac{1}{s_2-s_1/r_1}\int_0^\infty (e^{-s_1u}-e^{-s_2r_1u})\phi_Z(u)\D u\nonumber\\
&=\frac{\mu_1+\mu_2/r_1}{s_2-s_1/r_1}\left(\frac{1}{\psi_Z(s_1)}-\frac{1}{\psi_Z(r_1s_2)}\right)\nonumber\\
&=\frac{\mu_2+\mu_1r_1}{s_1-r_1s_2}\left(\frac{1}{\psi_1(r_1s_2)+\mu_2s_2}-\frac{r_1}{r_1\psi_1(s_1)+\mu_2s_1}\right),\label{eq:tocompare}
\end{align}
where $\psi_Z$ is the Laplace exponent of~$Z$.

Let us check this result against Theorem~\ref{thm:main}.
Note that $\mu_2=c_2>0$, $\psi_2(s_2)=\mu_2 s_2,\Phi_2(q)=q/\mu_2$ and so $\psi^Y_2(q)=0$ implying that $Y_2$ is a zero process.
Thus $\Psi^-_L(w)=\Psi^-_R(w)=1,$
\[\Psi^+_L(w)=\e e^{-w Y_1(r_1 e_{p_L})}=\e e^{\psi_1^Y(w)r_1e_{p_L}}=\frac{p_L}{p_L-r_1\psi_1^Y(w)}=\frac{p_L}{\mu_2+r_1w/\Phi_1(w)}\]
and similarly $\Psi^+_R(w)=p_R/(r_2\mu_2+w/\Phi_1(w))$.
Hence for $s_1,s_2$ large enough we have
\[\frac{\Psi_L^+(\psi_1(s_1))}{\Psi_R^+(\psi_1(s_1))}=\frac{p_L}{p_R}\frac{r_2\mu_2s_1+\psi_1(s_1)}{\mu_2s_1+r_1\psi_1(s_1)}.\]
According to Theorem~\ref{thm:main} we can write
\begin{align*}
(\psi_1(s_1)+\mu_2s_2)F(s_1,s_2)=
\frac{\mu_2p_R'}{\mu_2s_1+r_1\psi_1(s_1)}\frac{p_L}{p_R}
&+\frac{r_1\mu_1+\mu_2}{s_1-r_1s_2}\frac{\psi_1(s_1)-\psi_1(r_1s_2)}{\mu_2s_2+\psi_1(r_1s_2)},
\end{align*}
which indeed coincides with~\eqref{eq:tocompare}, because $p_R'p_L/p_R=r_1\mu_1+\mu_2$.
The latter is established by checking the two cases $\mu_1\geq 0$ and $\mu_1<0$ separately.

\subsection{Example: $r_1r_2=1$}\label{sec:nocost}
Assume that $r_1r_2=1$ which in particular holds when there are no transaction costs, i.e.\ $r_1=r_2=1$.
The net profit condition reduces to a single inequality $\mu_1+r_2\mu_2>0$.
From~\eqref{eq:def} it follows that
\[S_1(t)+r_2S_2(t)=u+r_2v+X_1(t)+r_2X_2(t)+E_1(t)+r_2E_2(t).\]
Observe that the ruin time $\tau$ is the first jump of $E_1(t)+r_2E_2(t)$.
Hence we have $u+r_2v+X_1(\tau)+r_2X_2(\tau)<0$ and $u+r_2v+X_1(t)+r_2X_2(t)\geq 0$ for all $t<\tau$.
Therefore, $\tau$ is the ruin time of the classical Cram\'er-Lundberg process $Z(t)=X_1(t)+r_2X_2(t)$ started in $u+r_2v=u+v/r_1$.
Similar to~\eqref{eq:tocompare} we get
\begin{align}\label{eq:nocosts}
&F(s_1,s_2)=\int_{\R_+}e^{-s_1u-s_2v}\phi_Z(u+v/r_1)\D u\D v\nonumber\\
&=\frac{\mu_1+\mu_2/r_1}{s_2-s_1/r_1}\left(\frac{1}{\psi_Z(s_1)}-\frac{1}{\psi_Z(r_1s_2)}\right)\nonumber\\
&=\frac{\mu_1+r_2\mu_2}{s_2-s_1r_2}\left(\frac{1}{\psi_1(s_1)+\psi_2(r_2s_1)}-\frac{1}{\psi_1(r_1s_2)+\psi_2(s_2)}\right).
\end{align}

Let us now check this result against Theorem~\ref{thm:main}.
Firstly, note that $p_R=r_2p_L$ and $X_R(t)=X_L(r_2t)$, which implies that the Wiener-Hopf factors for both processes are the same:
$\Psi_L^+(w)=\Psi_R^+(w)$ and
$\Psi_L^-(w)=\Psi_R^-(w)$.
Hence we have
\begin{align*}
&(\psi_1(s_1)+\psi_2(s_2))F(s_1,s_2)\\
&=\frac{p_R'}{s_2-r_2s_1}\frac{\psi_2(s_2)-\psi_2(r_2s_1)}{\psi_1(s_1)+\psi_2(r_2s_1)}
+\frac{p_L'}{s_1-r_1s_2}\frac{\psi_1(s_1)-\psi_1(r_1s_2)}{\psi_2(s_2)+\psi_1(r_1s_2)}\\
&=\frac{p_R'}{s_2-r_2s_1}\frac{\psi_1(s_1)+\psi_2(s_2)}{\psi_1(s_1)+\psi_2(r_2s_1)}
+\frac{p_L'}{s_1-r_1s_2}\frac{\psi_1(s_1)+\psi_2(s_2)}{\psi_2(s_2)+\psi_1(r_1s_2)}
-\frac{p_R'}{s_2-r_2s_1}-\frac{p_L'}{s_1-r_1s_2}.
\end{align*}
Note that $p_R'=r_2p_L'=\mu_1+r_2\mu_2$ and so we indeed obtain~\eqref{eq:nocosts}.

\subsection{Numerical experiments}
In this section we provide a simple numerical illustration of our main result, Theorem~\ref{thm:main}, by computing 
\[\hat F(s_1,s_2)=s_1s_2F(s_1,s_2)=\e\phi(e_{s_1},e_{s_2}),\]
i.e.\ the survival probability for initial capitals $e_{s_1},e_{s_2}$, which are two independent exponential random variables of rates $s_1$ and $s_2$ respectively.   
The main difficulty lies in evaluating the Wiener-Hopf factors. This task can be accomplished using various approaches, see e.g.~\cite{denIsiger} presenting an algorithm based on the Spitzer identity and Laplace inversion technique. 
In this study, however, we use a Monte Carlo simulation algorithm to obtain Wiener-Hopf factors, which is relatively straightforward to implement since the processes $X_L$ and $X_R$ are CPPs. These two processes are constructed from the inverse local time processes $Y_1$ and $Y_2$, which we discuss in the following.

Consider any of $Y_1,Y_2$ and for the moment drop the index to simplify notation.
From the general theory of local times, see~\cite[Ch.\ 6.2]{kyprianou}, we know that $Y$ is a CPP with jump distribution $\p_0(\tau_0^-\in\D x|\tau_0^-<\infty)$, where $\tau_0=\inf\{t\geq 0:X(t)<0\}$. Let us determine the jump rate $\lambda_Y$ of $Y$ assuming that $\psi'(0)=\mu>0$:
\begin{align*}
\psi^Y(q)&=\lambda_Y(\e(e^{-q\tau_0^-}|\tau_0^-<\infty)-1)=
\lambda_Y\left[\left(1-\frac{q}{\Phi(q)c}\right)/\left(1-\frac{\psi'(0)}{c}\right)-1\right]\\
&=\lambda_Y\frac{\Phi(q)\mu-q}{\Phi(q)(c-\mu)}=\frac{\lambda_Y}{c-\mu}\left(\mu-\frac{q}{\Phi(q)}\right),
\end{align*}
where we used the exit identity~\cite[(8.6)]{kyprianou} and the fact that $W^{(q)}(0)=W^{(0)}(0)=1/c$ in that formula. Comparing to~\eqref{eqY} we see that $\lambda_Y=c-\mu$. This allows to simulate the process $Y$ by drawing from $\p_0(\tau_0^-\in\D x|\tau_0^-<\infty)$, which in turn requires to simulate the original process and to reject the paths with `long survival'.

In our experiment we take deterministic claim sizes $C_1=C_2=1$, and put $c_1=c_2=1,\lambda_1=0.5,\lambda_2=0.9$ and $r_1=r_2=1.1$, which ensures that $\mu_1,\mu_2>0$ and yields $\psi_i(s)=s-\lambda_i(1-e^{-s})$.
First we simulate $N=10,000$ realizations of $\overline X_L(e_{p_L}),\overline X_R(e_{p_R}),\underline X_L(e_{p_L}),\underline X_R(e_{p_R})$ and then assign values to the Wiener-Hopf factors in an obvious way, e.g.\ $\Psi^+_L(w)=\sum_i \exp\left(-w \overline X_L^i(e_{p_L}\right)/N$, see Figure~\ref{fig:WH}.
\begin{figure}[h!]
\centering
\caption{Simulated Wiener-Hopf factors $\Psi_R^+(w),\Psi_L^+(w),\Psi_L^-(-w),\Psi_R^-(-w)$ (from top to bottom).}
\includegraphics[width=0.5\textwidth]{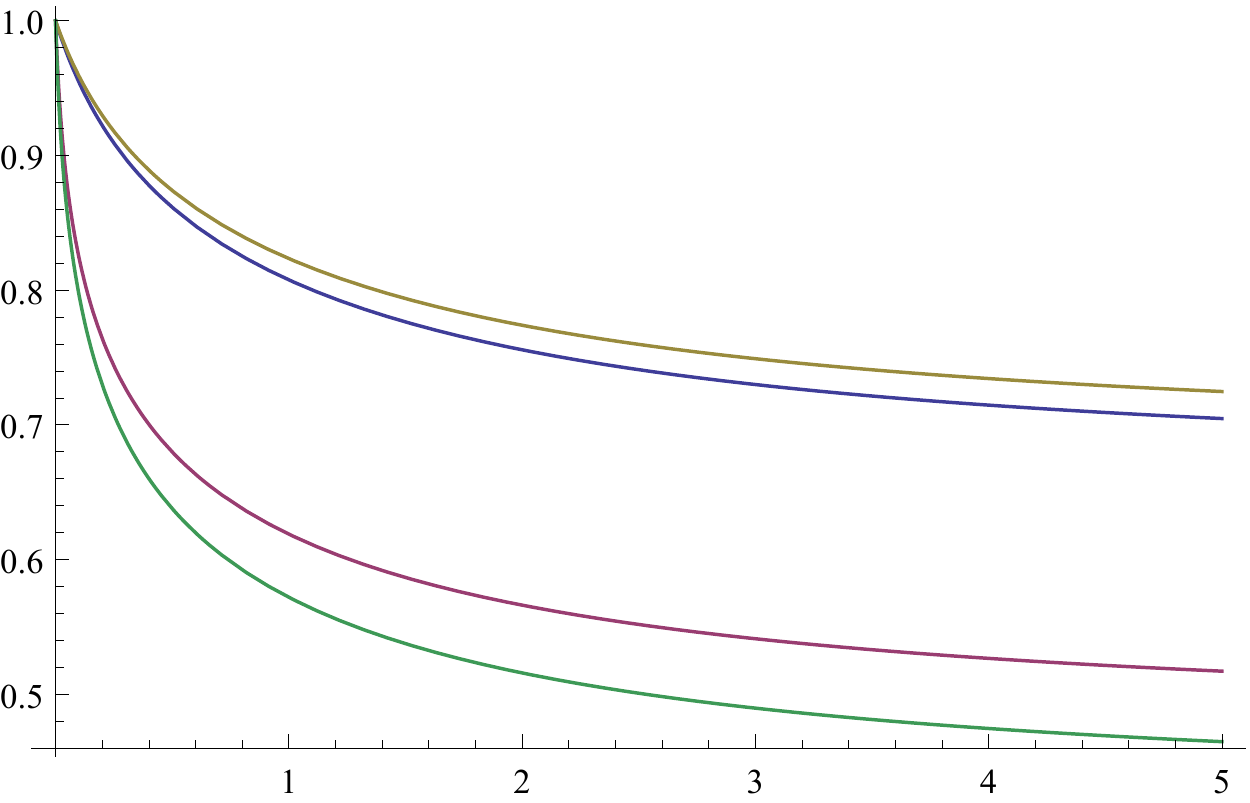}
\label{fig:WH}
\end{figure}
Now we can use Theorem~\ref{thm:main} to plot $\e\phi(e_{s_1},e_{s_2})$,
see Figure~\ref{fig:numerics} where we fix $s_2=1$ and let $s_1\in[1,10]$. 
Note that the survival probability is a decreasing function of $s_1$ and thus it is increasing as a function of the mean initial capital $\e e_{s_1}=1/s_1$.
\begin{figure}[h!]
\centering
\caption{Probability of survival $\e\phi(e_{s_1},e_{s_2})$ for exponential initial capitals as a function of $s_1$ with $s_2=1$. Dots represent direct simulation of the survival probability not using Theorem~\ref{thm:main}.}
\includegraphics[width=0.5\textwidth]{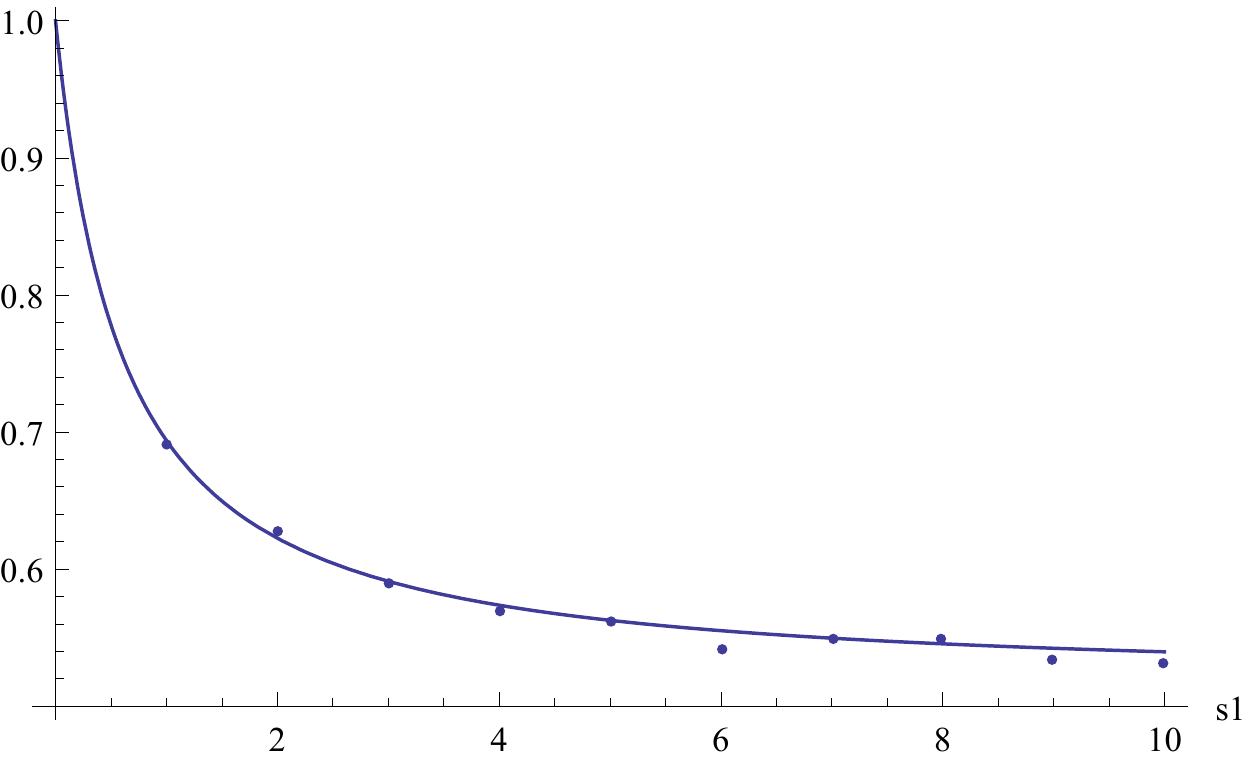}
\label{fig:numerics}
\end{figure}
Finally, we report $\phi(0,0)$, the survival probability from zero initial capitals. Corollary~\ref{cor:zero_init} yields $0.281$ and $0.277$ using positive and negative Wiener-Hopf factors respectively; the direct simulation gives $0.279$.

\section{Making some capital transfers impossible}\label{sec:restricted}
 One can modify our risk model so that {\em only the second company is allowed to help the first one, but not the opposite way}.
  In other words, we regard the second company as an insurer of the first (against deficit) with its own stream of claims.
Note that this new model is governed by the equations
 \begin{align*}
S_1(t)=u+X_1(t)+L_1(t)+E_1(t),\nonumber\\
S_2(t)=v+X_2(t)-r_1L_1(t)+E_2(t).
\end{align*}
In the following we specialize our results to this case by taking limits as $r_2\rightarrow\infty$.
The survival probability we thus obtain corresponds to survival of both insurer and reinsurer.

Firstly, the kernel equation becomes
\begin{align*}
\psi(s_1,s_2)F(s_1,s_2)=c_2F_1(s_1)+\frac{\psi(s_1,s_2)-\psi(r_1s_2,s_2)}{s_1-r_1s_2}F_2(s_2),
\end{align*}
which is immediate from Proposition~\ref{prop:kernel} and \eqref{eq:transform}.
Secondly, the net profit condition now reads
\begin{align}\label{eq:net2}
&\mu_2>0, &\mu_2+r_1\mu_1>0,
\end{align}
which is seen by repeating the steps of the proof of Proposition~\ref{prop:phi}.
Assuming independence of claim streams, we obtain the following corollary of Theorem~\ref{thm:main}.
\begin{cor}\label{cor:restricted}
Assume that $r_2=\infty$ and $X_1,X_2$ are independent and satisfy~\eqref{eq:net2}.
Then for $s_i>\Phi_i(0)$ it holds that
  \begin{align*}
&(\psi_1(s_1)+\psi_2(s_2))F(s_1,s_2)=(\mu_2-r_1\mu_1^-)\frac{\Psi_L^+(\psi_1(s_1))}{s_1}\\
&+\mu_2(\mu_2+r_1\mu_1)\frac{s_2}{s_1-r_1s_2}\frac{\psi_1(s_1)-\psi_1(r_1s_2)}{\psi_2(s_2)+\psi_1(r_1s_2)}\frac{1}{\psi_2(s_2)\Psi_L^-(-\psi_2(s_2))}.
\end{align*}
\end{cor}
\begin{proof}
Let us examine the limiting quantities (as $r_2\rightarrow\infty$) from the statement of Theorem~\ref{thm:main}.
Firstly, $\Psi_L^\pm$ stay the same, but
\begin{align*}
&\Psi_R^+(w)\rightarrow 1, &\Psi_R^-(w)\rightarrow-\frac{\mu_2\Phi_2(-w)}{w}.
\end{align*}
To see this notice that $p_R\rightarrow\infty$ and so $\overline X_R(e_{p_R})\leq Y_1(e_{p_R})\rightarrow 0$ a.s., which yields the first limit.
The second is then obtained from~\eqref{eq:WH} (written for `R') and~\eqref{eq:phi_phi2}.
Finally, observe that $p_L'=r_1\mu_1+\mu_2, p_R'/r_2\rightarrow \mu_2-r_1\mu_1^-$ and take the limit in the equation of Theorem~\ref{thm:main}.
\end{proof}

Finally, let us assume that neither company can help the other, i.e.\ $r_1=r_2=\infty$.
Thus we retrieve the standard bivariate `or' problem, where the ruin means ruin in at least one line, see~\cite[Eq.\ (1.5)]{cai2005multivariate}.
The kernel equation becomes
\[\psi(s_1,s_2)F(s_1,s_2)=c_2F_1(s_1)+c_1F_2(s_2).\]
Now assume that $\mu_1,\mu_2>0$ and that $X_1,X_2$ are independent, and take the limit as $r_1\rightarrow\infty$
in Corollary~\ref{cor:restricted}. Noticing that
\begin{align*}
&\Psi_L^+(w)\rightarrow \frac{\mu_1\Phi_1(w)}{w}, &\Psi_L^-(w)\rightarrow1.
\end{align*}
we obtain
\[(\psi_1(s_1)+\psi_2(s_2))F(s_1,s_2)=\frac{\mu_1\mu_2}{\psi_1(s_1))}+\frac{\mu_1\mu_2}{\psi_2(s_2)}.\]
This yields the correct product formula
\[F(s_1,s_2)=\frac{\mu_1}{\psi_1(s_1))}\frac{\mu_2}{\psi_2(s_2))}\]
providing yet another check.

\section{Open problems}
\label{Open}
Various interesting directions for future work exist with respect to the present model.
Firstly, one may try to extend the model to allow for more general driving processes, that is for {\em \levy processes with negative jumps}.
This seems to be a hard problem contrary to many other settings, where such generalizations require little additional effort.
In fact, there are essential difficulties in each step leading to Theorem~\ref{thm:main}:
\begin{itemize}
 \item \emph{Definition of the model}. In general one can not use the recursive definition of Section~\ref{sec:recursive}.
The definition in Section~\ref{sec:skorokhod} (or a variant of it) seems to be a natural one. Nevertheless one has to establish that there is a unique solution to this modified Skorokhod problem.
\item \emph{The kernel equation}. We believe that Proposition~\ref{prop:kernel} holds for a general bivariate \levy process $(X_1,X_2)$ with negative jumps.
One may try to employ the generator of the underlying Markov process instead of an $o(h)$ reasoning.
There are various problems on this way including the required smoothness of the survival function.
Alternatively, one may try to use a martingale approach, but the choice of an appropriate martingale is far from trivial.
\item \emph{Identification of boundary measures}. One may follow the ideas from the proof of Theorem~\ref{thm:main}.
In the L\'evy case, a highly problematic step is to bound both sides of~\eqref{eq:entire} by a constant.
\end{itemize}

Secondly, it is important to understand if the kernel equation in Proposition~\ref{prop:kernel} characterizes the survival probability in some sense.
In the case of a reflected Brownian motion in a quadrant the basic adjoint relation characterizes the stationary distribution together with the boundary measures,
see~\cite[Thm.\ 3.4]{williams_survey}.
This property allowed Dai and Harrison~\cite{dai_numerical} to construct an algorithm computing the stationary distribution based on the basic adjoint relation.

Finally, one may try to interpret Theorem~\ref{thm:main} (and the quite similar Theorem~1 in~\cite{boxma_ivanovs}) to provide a probabilistic approach.
Moreover, one may consider some alternative risk measures of the model in~\eqref{eq:def} such as discounted external injections of capital.
One may also try to
extend the model to a setting with multiple companies. It seems that a probabilistic approach could be very helpful in this respect.

\noindent
{\bf Acknowledgment}
\\
The authors are indebted to Hansjörg Albrecher (Univ. Lausanne), Esther Frostig and David Perry (Univ. Haifa) for stimulating discussions.
The research of Onno Boxma was supported by the TOP-I grant {\em Two-dimensional Models in Queues and Risk} of the Netherlands Organisation for Scientific Research (NWO),
and Jevgenijs Ivanovs was supported by the Swiss National Science Foundation Project 200020\_143889.

\section*{Appendix}
\begin{proof}[Proof of~\eqref{eq:lh}]
Using monotonicity of $\phi$ in both arguments we see that the rhs of~\eqref{eq:lh}
provides a lower bound for $\phi(u,v)$.
The upper bound follows in a similar way.
\begin{align*}
\phi(u,v)\leq o(h)&+(1-\lambda h)\phi(u+c_1h,v+c_2h)\\
&+\l h \int_0^{u}\int_0^{v}\phi(u+c_1h-x,v+c_2h-y)\mu(\D x,\D y)\\
&+\l h \int_0^{u}\int_{v}^{\infty}\phi(u+c_1 h-x+r_2(v-y+c_2h),c_2h)\mu(\D x,\D y)
\\&+\l h \int_{u}^\infty\int_0^{v}\phi(c_1 h,v+c_2h-y+r_1(u-x+c_1 h))\mu(\D x,\D y).
\end{align*}
According to~\eqref{eq:bound} the difference of the bounds is of order $o(h)$, which completes the proof.
\end{proof}


\end{document}